\documentclass[runningheads]{llncs}
\usepackage{graphicx}
\usepackage{amsfonts}
\usepackage{amsmath,amssymb}
\usepackage{algorithm}
\usepackage{algpseudocode}
\usepackage{pifont}
\usepackage{float}
\usepackage{graphicx}
\usepackage{caption}
\usepackage{subcaption}
\usepackage{blindtext}

\newcommand{\circledOne}{\text{\ding{172}}}
\newcommand{\circledTwo}{\text{\ding{173}}}
\newcommand{\circledThree}{\text{\ding{174}}}
\newcommand{\circledFour}{\text{\ding{175}}}
\newcommand{\circledFive}{\text{\ding{176}}}
\newcommand{\circledSix}{\text{\ding{177}}}
\newcommand{\circledSeven}{\text{\ding{178}}}
\newcommand{\circledEight}{\text{\ding{179}}}
\newcommand{\circledNine}{\text{\ding{180}}}

\DeclareMathOperator*{\argmin}{arg\,min}

\newcommand{\dotprod}[2]{\left\langle #1,#2 \right\rangle}
\newcommand{\norms}[1]{\left\| #1 \right\|}
\newcommand{\expect}[1]{\mathbb{E}\left[ #1 \right]}

\renewcommand{\gg}{\mathbf{g}}
\newcommand{\bb}{\mathbf{b}}

\usepackage[breaklinks=true,colorlinks=true,linkcolor=red, citecolor=blue, urlcolor=blue]{hyperref}%

\usepackage[backend=biber,bibencoding=utf8,sorting=none,style=gost-numeric,language=autobib,autolang=other,clearlang=true,sortcites=true,doi=false,isbn=false,date=year]{biblatex}

\usepackage{xcolor}
\newcommand{\al}[1]{{\color{black}#1}} 

\begin{document}
\title{Accelerated Zero-Order SGD Method for Solving the Black Box Optimization Problem under ``Overparametrization'' Condition\thanks{The research was supported by Russian Science Foundation (project No. 21-71- 30005), \url{https://rscf.ru/en/project/21-71-30005/}.}}
\titlerunning{AZO-SGD Method for Overparametrization Problems}
%

\author{Aleksandr Lobanov\inst{1,2}\orcidID{0000-0003-1620-9581}  \and
Alexander Gasnikov\inst{1,2,3}\orcidID{0000-0002-7386-039X}}

\authorrunning{A. Lobanov et al.}
%
\institute{Moscow Institute of Physics and Technology, Dolgoprudny, Russia
 \and 
 ISP RAS Research Center for Trusted Artificial Intelligence, Moscow, Russia 
\and 
Institute for Information Transmission Problems RAS, Moscow, Russia \\
\email{\{lobanov.av,gasnikov.av\}@mipt.ru}}
\maketitle              
\begin{abstract}
This paper is devoted to solving a convex stochastic optimization problem in a overparameterization setup for the case where the original gradient computation is not available, but an objective function value can be computed. For this class of problems we provide a novel gradient-free algorithm, whose creation approach is based on applying a gradient approximation with $l_2$ randomization instead of a gradient oracle in the biased Accelerated SGD algorithm, which generalizes the convergence results of the AC-SA algorithm to the case where the gradient oracle returns a noisy (inexact) objective function value. We also perform a detailed analysis to find the maximum admissible level of adversarial noise at which we can guarantee to achieve the desired accuracy. We verify the theoretical results of convergence using a model example.

\end{abstract}

\section{Introduction} \label{sec:Introduction}

The black-box optimization problem \cite{Audet_2017,Conn_2009, Pardalos_2021} usually arises when we have few information about the objective function. Such problems can appear when the objective function is non-smooth \cite{Gasnikov_2022_ICML} (i.e. no gradient computation is available) or, for instance, when the function is smooth \cite{Akhavan_2022} (may even have a higher order of smoothness \cite{Bach_2016}), but the process of computing the derivatives is too expensive in contrast to computing the value of the objective function. Moreover, this problem is often solved under conditions of privacy, when some data cannot be disseminated due to confidentiality. Then a gradient-free oracle \cite{Rosenbrock_1960} (or in other words zero-order oracle) acts as some ``black box", which returns only objective function value $f(x)$ at requested point $x$ with some bounded adversarial noise $\delta(x) \leq \Delta$ (where $\Delta$ is level noise). Where the latter means that the gradient-free oracle can give an inexact value of the objective function. As practice shows~\cite{Bogolubsky_2016}, the lower the level of adversarial noise, the more expensive it is to call a gradient-free oracle, so it is important to understand how large the level of adversarial noise can be, at which a ``good" convergence rate is still guaranteed (by ``good" convergence rate is meant the convergence of the algorithm when $\Delta = 0$). This concept of a gradient-free oracle is common in the literature and has the interpretation of a adversarial attack on the black-box model \cite{Chen_2017}. The black box problem is currently actively researched in the optimization \cite{Duchi_2015, Shibaev_2022} and machine learning \cite{Papernot_2016, Papernot_2017} community, since this problem has applications in the following areas: federated learning \cite{Lobanov_2022, Patel_2022}, distributed learning \cite{Scaman_2019, Lobanov_2023_SADOM} and deep learning~\cite{Gao_2018}. A particular need for solving such problems arises in the following applications: model hyperparameters tuning \cite{Hazan_2017,Elsken_2019}, reinforcement learning \cite{Choromanski_2018, Mania_2018}, multi-armed bandits \cite{Bartlett_2008, Bubeck_2012, Shamir_2017}, and many others. 

There are various approaches to solving the black-box problem, but the most common in a theoretically proof-of-concept sense is to create gradient-free optimization algorithms based on the state of the art first-order methods and using various randomized gradient approximations \cite{Gasnikov_2022}. The most common first-order optimization methods in machine learning models are momentum SGD, Adam, and others. But note that these algorithms are variants of Stochastic Gradient Descent (SGD) \cite{Robbins_1951, Bottou_2018}, which use stochastic gradient estimates. By adding a procedure for batching stochastic gradient estimates, it is possible to obtain algorithms that are easily parallelized on several computers. It is data distribution (parallelization) that significantly reduces the computational costs that certainly arise in a huge number of modern machine learning models. Also, with the addition of acceleration one can still achieve improvement in terms of estimates on the number of successive iterations. Thus, for many optimization problems, the state of the art first-order algorithms are accelerated batched variants of SGD. 

In this paper, we focus on solving a stochastic convex black-box optimization problem in a smooth setting with an overparameterization condition. Where the latter means that the model has many more parameters than the data available. We can summarize our contributions as follows:
\begin{itemize}
    \item We derive the convergence rate for a biased Accelerated Stochastic Gradient Descent, which covers smooth convex stochastic optimization problems under the overparameterization setup.

    \item We provide a novel Accelerated Zero-Order Stochastic Gradient Descent Method (AZO-SGD) for solving the black-box problem in a smooth setting under the overparameterization condition. We analyze the robustness of AZO-SGD algorithm to adversarial noise, providing an estimate for the maximum admissible level of adversarial noise at which the desired accuracy can still be achieved. We show that our algorithm is optimal on oracle calls in the class of gradient-free algorithms.

    \item We show the convergence of the Accelerated Zero-Order Stochastic Gradient Descent Method proposed in this paper using a model example of finite sums in which the number of summands is less than the number of variables (the overparameterization condition). 
\end{itemize}

\subsection{Related works}

\paragraph{Adversarial noise.}

Finding the maximum admissible noise level at which one can still guarantee convergence to the desired accuracy $\varepsilon$ is an important issue for the black-box optimization problem. Special attention to this question was allocated by the works \cite{Lobanov_2022, Dvinskikh_2022, Kornilov_2023, Lobanov_2023_WAIT}. For example, in \cite{Dvinskikh_2022} the authors found the maximum admissible level of \textit{adversarial deterministic noise} for a non-smooth convex black-box optimization problem $ \sim \mathcal{O} \left(\varepsilon^2 d^{-1/2} \right)$, moreover in \cite{Lobanov_2022} it is shown that this estimate will be the same for $l_1$ and $l_2$ randomization. In \cite{Kornilov_2023}, the authors showed that by assuming a strong convexity, the estimate maximum level of \textit{adversarial deterministic noise} can be improved to ~$\sim~\mathcal{O} \left(\mu^{1/2} \varepsilon^{3/2} d^{-1/2} \right)$ in non-smooth setting. And in \cite{Lobanov_2023_WAIT} the authors were able to show that this estimation in a non-smooth one can be also improved to~$\sim~\mathcal{O} \left(\varepsilon d^{-1/2} \right)$ by using \textit{adversarial stochastic noise}, since this concept of \textit{adversarial stochastic noise} does not accumulate in the bias, but accumulates only in the variance. In addition, this paper shows that if the function is smooth, the estimate of the maximum level of adversarial stochastic noise can be improved to~$\sim~\mathcal{O} \left(\varepsilon^{1/2} d^{-1/2} \right)$.  In this paper, we will use a gradient approximation via $l_2$ randomization to create a novel gradient-free algorithm, and we will find the maximum admissible level of deterministic noise in a smooth setting under overparameterization condition.

\paragraph{SGD type algorithms.}

Many works \cite{Lan_2012, Woodworth_2021, Gorbunov_2020, Gasnikov_2016, Ajalloeian_2020, Polyak_1963, Lojasiewicz_1963, Yue_2022, Woodworth_2021_over} study the Stochastic Gradient Descent and its variant in different setups. For example, in \cite{Lan_2012} the authors proposed an accelerated method of stochastic gradient descent, AC-SA. Later in \cite{Woodworth_2021} authors proposed optimal algorithms for federated learning architecture, which is based on AC-SA (Single-Machine Accelerated SGD and Mini-Batch Accelerated SGD) method. In \cite{Gorbunov_2020} proposed an clipped accelerated SGD method for heavy-tailed optimization problems based on the accelerated SGD method: Stochastic Similar Triangle Method (SSTM) \cite{Gasnikov_2016}. In \cite{Ajalloeian_2020}, the authors studied the biased SGD method in the Polak-Lojasiewicz \cite{Polyak_1963, Lojasiewicz_1963} setup. It is worth noting that in \cite{Yue_2022} it was shown that for problems satisfying the Polak-Lojasiewicz condition the non-accelerated SGD algorithm will be optimal. In \cite{Woodworth_2021_over}, the study of the AC-SA (accelerated SGD) algorithm was continued already in the overparameterization setup. We in this paper generalize the analysis of the AC-SA algorithm from \cite{Woodworth_2021_over}, to create a biased accelerated SGD algorithm in the overparameterization setting. A biased first-order algorithm is necessary because using $l_2$ randomization produces a bias in the case when $\delta(x) > 0$. Therefore, based on the biased batched accelerated stochastic gradient descent and using $l_2$ randomization, we create a new gradient-free optimization algorithm to solve a convex stochastic black-box optimization problem under overparameterization setup.

\paragraph{Gradient noise assumptions.}

Recently there is a trend in works \cite{Woodworth_2021_over, Rakhlin_2012, Hazan_2014, Bertsekas_1996, Stich_2019, Lobanov_2023, Schmidt_2013, Srebro_2010} of relaxed stochastic gradient variance restriction condition. Very many works (e.g. see \cite{Rakhlin_2012, Hazan_2014}) use the standard assumption: $\expect{\norms{\nabla f(x,\xi)}^2} \leq \sigma^2$. However, already in the works \cite{Bertsekas_1996, Stich_2019, Lobanov_2023, Zhang_2022} used in the analysis of the algorithm a more relaxed assumption of weak growth: $\expect{\norms{\nabla f(x, \xi)}^2} \leq M\norms{\nabla f(x)}^2 + \sigma^2$. The following work \cite{Schmidt_2013} have set the constants so that the strong growth condition assumption is satisfied: $\expect{\norms{\nabla f(x, \xi)}^2} \leq M\norms{\nabla f(x)}^2$. In \cite{Woodworth_2021_over, Srebro_2010} the condition satisfying the overparameterized set: $\expect{\norms{\nabla f(x^*,\xi)}^2} \leq \sigma_*^2$. In this paper, we will also assume uniform smoothness of the function over $\xi$ as well as the overparameterized condition, since our approach to creating gradient-free algorithms is based on the Accelerated Batched Stochastic Gradient Descent (AC-SA) proposed in~\cite{Woodworth_2021_over}.

\subsection{Notations}
We use $\dotprod{x}{y}:= \sum_{i=1}^{d} x_i y_i$ to denote standard inner product of $x,y \in \mathbb{R}^d$, where $x_i$ and $y_i$ are the $i$-th component of $x$ and $y$ respectively. We denote Euclidean norm ($l_2$-norm) in~$\mathbb{R}^d$ as $\norms{x} = \| x\|_2 := \sqrt{\dotprod{x}{x}}$. We use the following notation $B_2^d(r):=\left\{ x \in \mathbb{R}^d : \| x \| \leq r \right\}$ to denote Euclidean ball ($l_2$-ball) and  $S_2^d(r):=\left\{ x \in \mathbb{R}^d : \| x \| = r \right\}$ to denote Euclidean sphere. Operator $\mathbb{E}[\cdot]$ denotes full mathematical expectation. We notation $\tilde{O} (\cdot)$ to hide logarithmic factors. 

\setcounter{footnote}{3}
\subsection{Paper Ogranization}
This paper has the following structure. In Section \ref{sec:Introduction} we introduce this paper and also provide related works. In Section \ref{sec:Technical_Preliminaries} we formulate the problem statement. While in Section \ref{sec:biased_gradient} we provide an accelerated SGD algorithm with a biased gradient oracle in the reparameterization setup. Section \ref{sec:Main_Result} presents the main result of this paper. We confirm the theoretical results via a model example in Section \ref{sec:Experiments}. While Section \ref{sec:Conclusion} concludes this paper. Detailed proofs are presented in the supplementary materials (Appendix)\footnote{The full version of this article, which includes the Appendix can be found by the article title in the arXiv at the following link: \url{https://arxiv.org/abs/2307.12725}.}.


\section{Technical Preliminaries} \label{sec:Technical_Preliminaries}
We study a standard stochastic convex optimization problem:
\begin{equation}\label{eq:init_problem}
    f^* = \min_{x \in \mathbb{R}^d}\left\{ f(x) := \expect{f(x, \xi)} \right\},
\end{equation}
where $f: \mathbb{R}^d \rightarrow \mathbb{R}$ is smooth convex function that we want to minimize over $\mathbb{R}^d$. This problem statement is a general smooth convex stochastic optimization problem, so to define the class of the problem considered in this paper we will introduce some assumptions on the objective function and on the gradient oracle. 

\subsection{Assumptions on the Objective Function}
In all proofs we assume convexity and smoothness of the function $f(x,\xi)$.
\begin{assumption} \label{ass:convex}
    For almost every $\xi$, $f(x,\xi)$ is non-negative, convex w.r.t. $x$,~i.e.
    \begin{equation*}
        \forall x, y, \xi \quad f(y, \xi) \geq f(x, \xi) + \dotprod{\nabla f(x, \xi)}{y - x}.
    \end{equation*}
\end{assumption}

\begin{assumption} \label{ass:smooth}
    For almost every $\xi$, $f(x,\xi)$ is non-negative, $L$-smooth w.r.t.~$x$,~i.e.
    \begin{equation*}
        \forall x, y, \xi \quad f(y, \xi) \leq f(x, \xi) + \dotprod{\nabla f(x, \xi)}{y - x} + \frac{L}{2} \norms{y - x}^2.
    \end{equation*}
\end{assumption}
Assumptions \ref{ass:convex} and \ref{ass:smooth} are common in the literature (see, e.g., \cite{Vaswani_2019, Fatkhullin_2022}). However, it is worth noting that these assumptions require uniform convexity and smoothness over $\xi \sim \mathcal{D}$ \cite{Tran_2022}. This is an essential difference from the standard assumptions, which define a narrower class of the objective function.

\begin{assumption} \label{ass:f_star}
    The function $f(x)$ is a convex and has the minimum value $f^* = \min_{x} f(x)$, which is attained at a point $x^*$ with $\norms{x^*} \leq R$.
\end{assumption}
We explicitly introduce the problem solution $f^*$ in Assumption \ref{ass:f_star}, since our approach implies that the convergence rate depends on the solution $f^*$ (see, e.g., \cite{Cotter_2011}), i.e., our analysis will show an improvement in convergence at $f^* \rightarrow 0$.

\subsection{Assumptions on the Gradient Oracle}
In our analysis we consider the case when we obtain an inexact gradient value when calling the oracle. Therefore we first define a biased gradient oracle.
\begin{definition}[Gradient Oracle] \label{def:biased_oracle}
A map $\gg: \mathbb{R}^d \times \mathcal{D} \rightarrow \mathbb{R}^d$ s.t.
\begin{equation*}
    \gg(x,\xi) = \nabla f(x, \xi) + \bb(x),
\end{equation*}
where $\bb: \mathbb{R}^d \rightarrow \mathbb{R}^d$ such that $\forall x \in \mathbb{R}^d$: $\norms{\bb(x)}^2 \leq \zeta^2$. 
\end{definition} 
Next, we assume that gradient noise is bounded as follows.
\begin{assumption} \label{ass:stoch_noise}
    There exists $\sigma^2_* \geq 0$ such that $\forall x \in \mathbb{R}^d$
    \begin{equation*}
        \expect{\norms{\nabla f(x^*, \xi) - \nabla f(x^*)}^2} \leq \sigma^2_*.
    \end{equation*}
\end{assumption}
Assumption \ref{ass:stoch_noise} is a common assumption for overparameterized optimization problems, since in this setup in many problems \cite{Jacot_2018, Allen_2019, Belkin_2019} $f^*$ can be expected to be small. We introduced Assumptions \ref{ass:convex}--\ref{ass:stoch_noise} because in this paper we based on the results of~\cite{Woodworth_2021_over}, in which it was proved that the convergence rate of the AC-SA method (from \cite{Lan_2012}) has the following form in overparameterization setup~of~problem~\eqref{eq:init_problem}:
\begin{equation} \label{eq:woodworth_ACSA}
    \expect{f(x_{N}^{ag}) - f^*} \leq c \cdot \left( \frac{L R^2}{N^2} + \frac{L R^2}{BN} + \sqrt{\frac{L R^2 f^*}{BN}} \right),
\end{equation}
where $N$ is a iteration number, $B$ is a batch size and $\sigma^2_* \leq 2 Lf^*$ (it's proven~\cite{Woodworth_2021_over}).

\section{Accelerated SGD with Biased Gradient} \label{sec:biased_gradient}
Our approach to create a gradient-free algorithm (for a overparameterized setup) implies the use of $l_2$ randomization (see Section \ref{sec:Main_Result}), based on the AC-SA algorithm from \cite{Woodworth_2021_over}. However, the standard concept of a gradient-free oracle implies the presence of adversarial noise, which can accumulate in both variance and bias. Therefore, it is important to investigate the adversarial noise for the question of maximum admissible noise level, when the desired accuracy can be guaranteed. As can be seen, the result \eqref{eq:woodworth_ACSA} obtained in \cite{Woodworth_2021_over} does not account for the bias in the gradient oracle (see Definition \ref{def:biased_oracle} in the case when $\norms{\bb(x)} = 0$). Consequently, in Theorem \ref{th:biased_AC_SA} we provide a novel biased AC-SA algorithm that is robust to the overparameterized setup and accounts for bias in~gradient~oracle.

\begin{theorem}[Convergence of Biased AC-SA] \label{th:biased_AC_SA}
    Let $f$ satisfy Assumptions \ref{ass:convex}--\ref{ass:f_star} and gradient oracle from Definition \ref{def:biased_oracle} satisfy Assumption \ref{ass:stoch_noise}, then Biased AC-SA algorithm guarantees the convergence with a universal constant $c$
    \begin{equation*} 
        \expect{f(x_{N}^{ag}) - f^*} \leq c \cdot \left( \frac{L R^2}{N^2} + \frac{L R^2}{BN} + \frac{\sigma_* R}{\sqrt{BN}} + \zeta R + \frac{\zeta^2}{2 L}N \right).
    \end{equation*}
\end{theorem}
The results of Theorem \ref{th:biased_AC_SA} show the convergence of the AC-SA algorithm, considering the bias $\bb(x)$ in the gradient oracle.  It is not difficult to see that if we consider the case without bias ($\zeta = 0$), the convergence result of Theorem \ref{th:biased_AC_SA} will fully correspond to the result \eqref{eq:woodworth_ACSA}. The last two terms in Theorem \ref{th:biased_AC_SA} are standard for the accelerated algorithm (see, for example, \cite{Dvinskikh_2021, Vasin_2023}). There are several ways to obtain these results: using the ($\delta,L$)-oracle technique \cite{Devolder_2013}, modifying Assumptions \ref{ass:convex}, \ref{ass:smooth}, or performing sequential reasoning with current assumptions. The proof of Theorem \ref{th:biased_AC_SA} can be found in supplementary materials (Appendix~\ref{Appendix:proof_th1}).


\section{Main Result} \label{sec:Main_Result}
In this section, we present the main result of this work, namely a gradient-free algorithm for solving a convex smooth stochastic black-box optimization problem in an overparameterized setup. We further narrow the problem class \eqref{eq:init_problem} considered in this section to the black box problem, that is, when the calculation of the gradient oracle is not available for some reason. Unfortunately, we cannot apply the AC-SA algorithm or even the biased AC-SA algorithm to solving this problem class. Therefore, there is a need to create an algorithm that only requires calculations of function values. Such algorithms are usually called \textit{gradient-free}, since the efficiency of this class of algorithms is determined by three criteria: the maximum admissible level of adversarial noise $\Delta$, iterative complexity $N$, and in particular the total number of calls to the \textit{gradient-free oracle} $T$. Our approach in creating gradient-free algorithm based on the biased AC-SA algorithm. Instead of the gradient oracle (see Definition \ref{def:biased_oracle}) we use the gradient~approximation:
\begin{equation} \label{eq:gradient_approximation}
    \gg(x,\xi,e) = \frac{d}{2 \tau}\left( f_\delta(x + \tau e, \xi) - f_\delta(x - \tau e, \al{\xi}) \right) e,
\end{equation}
where $e$ is a vector uniformly distributed on unit sphere $S_2^d(1)$, $\tau$ is a smoothing parameter and $f_\delta(x, \xi) = f(x, \xi) + \delta(x)$ ($|\delta(x)| \leq \Delta$) is a \textit{gradient-free} oracle. Thus Algorithm \ref{alg:AZO_SGD} presents a novel gradient-free method, namely Accelerated Zero-Order Stochastic Gradient Descent (AZO-SGD) Method for solving the black-box optimization problem \eqref{eq:init_problem} under the overparameterization condition.

\begin{algorithm}
        \caption{Accelareted Zero-Order Stochastic Gradient Descent (AZO-SGD)}\label{alg:AZO_SGD}
        \textbf{Input}: Start point $x^{ag}_0 = x_0 \in \mathbb{R}^d$, maximum number of iterations $N \in \mathbb{Z}_+$.\\
        \hspace*{\algorithmicindent} Let stepsize $ \al{\gamma}_k > 0$, parameters $\beta_k, \tau > 0$, batch size $B \in \mathbb{Z}_+$. 
        \begin{algorithmic}[1]
        \For{$k=0,...,N-1$}
        \State $\beta_k = 1 + \frac{k}{6}$ and $\al{\gamma}_k = \al{\gamma} (k+1)$ for $\gamma = \min \left\{ \frac{1}{12 L}, \frac{B}{24 L(N+1)}, \sqrt{\frac{B R^2}{Lf^* N^3}} \right\}$  
        \State $x^{md}_k = \beta^{-1}_k x_k + (1 - \beta^{-1}_k) x_k^{ag}$
        \State Sample $\{ e_1,...,e_{B} \}$ and $ \{ \xi_1,...,\xi_{B} \}$ independently
        \State Define $\gg_k = \frac{1}{B} \sum_{i=1}^{B} \gg(x^{md}_k, \xi_i, e_i)$ using \eqref{eq:gradient_approximation}
        \State $\tilde{x}_{k+1} = x_k - \al{\gamma}_k \gg_k$
        \State $x_{k+1} = \min \left\{ 1, \frac{R}{\norms{\tilde{x}_{k+1}}} \right\} \tilde{x}_{k+1}$
        \State $x_{k+1}^{ag} = \beta^{-1}_k x_{k+1} + (1 - \beta^{-1}_k) x_{k}^{ag}$
        \EndFor
        \end{algorithmic}
        \textbf{Output}: $x_N^{ag}$.
    \end{algorithm}
    Next, we provide Theorem \ref{th:AZO_SGD}, in which we show the convergence results for Accelareted Zero-Order Stochastic Gradient Descent (AZO-SGD) Method.
    \begin{theorem} \label{th:AZO_SGD}
        Let $f$ satisfy Assumptions \ref{ass:convex}--\ref{ass:f_star} and gradient approximation \eqref{eq:gradient_approximation} with parameter $\tau \leq \frac{\varepsilon}{LR}$ satisfy Assumption \ref{ass:stoch_noise}, then Accelareted Zero-Order Stochastic Gradient Descent (AZO-SGD) Method (see Algorithm \ref{alg:AZO_SGD}) achieves $\varepsilon$-accuracy: $\expect{f(x_{N}^{ag}) - f^*} \leq \varepsilon$ after 
        \begin{equation*}
            N = \mathcal{O}\left( \sqrt{\frac{L R^2}{\varepsilon}} \right), \quad T = \max \left\{ \mathcal{O}\left( \frac{LR^2}{\varepsilon} \right), \mathcal{O}\left( \frac{d \sigma_*^2 R^2}{\varepsilon^{2}} \right) \right\}
        \end{equation*}
        number of iterations, total number of gradient-free oracle calls and at
        \begin{equation*}
            \Delta \leq \frac{\varepsilon^2}{d L R^2}
        \end{equation*}
        the maximum admissible level of adversarial noise.
    \end{theorem}
    The result of Theorem \ref{th:AZO_SGD} shows the effective iterative complexity $N$, since an accelerated method was taken as the base (biased AC-SA, see Theorem \ref{th:biased_AC_SA}). It is also worth noting that the batch size $B = \max \left\{ \mathcal{O}\left( \sqrt{\frac{L R^2}{\varepsilon}} \right), \mathcal{O}\left( \frac{d \sigma_*^2 R}{L^{1/2} \varepsilon^{3/2}} \right) \right\}$ can change with time, i.e., it directly depends on $\sigma^2_*$, which leads to an optimal estimate of the number of gradient-free oracle calls $T$. Finally, one of the main results of this paper is the estimation for the maximum admissible noise level $\Delta$. This estimation is inferior to the other estimations $ \sim \mathcal{O} \left(\varepsilon^2 d^{-1/2} \right)$ in the non-smooth setting and $ \sim \mathcal{O} \left(\varepsilon^{3/2} d^{-1/2} \right)$ in smooth setting, but this is expected, since Algorithm \ref{alg:AZO_SGD} is working in a different setup, namely in overparameterization condition. It's also noted that we don't guarantee that this evaluation is unimproved, but only states that at $ \Delta \leq \mathcal{O} \left(\varepsilon^2 d^{-1} \right)$ there will be convergence to $\varepsilon$-accuracy. We present an open question for future research: improving the estimate to the maximum admissible noise level, as well as finding upper bounds on noise level beyond which convergence in the overparameterized setup cannot~be~guaranteed. The proof of Theorem \ref{th:biased_AC_SA} can be found in supplementary~materials~(Appendix~\ref{Appendix:proof_th2}).

    \begin{remark}[General case]
        It is worth noting that this paper, and in particular Theorem \ref{th:biased_AC_SA} and Theorem \ref{th:AZO_SGD}, focus on the Euclidean case. However, using the work of \cite{Ilandarideva_2023} (namely, Algorithm 2) as a basis, and similarly generalizing the convergence results (Corollary 1, \cite{Ilandarideva_2023}) to the case with a biased gradient oracle (see Definition \ref{def:biased_oracle}), we can obtain a gradient-free algorithm for a more general class of problems ($L_p$-norm and presence of constraints). In this case, the parameters (number of iterations $N$, the maximum admissible level of adversarial noise $\Delta$, smoothing parameter $\tau$) of the gradient-free algorithm will be the same as in Theorem \ref{th:AZO_SGD}, except for the total number of oracle calls: let given $1/p + 1/q = 1$
        \begin{equation*}
            T = N \cdot B = \max \left\{ \mathcal{O}\left( \frac{LR^2}{\varepsilon} \right), \mathcal{O}\left( \frac{\min\{ q, \ln d \} d^{2 - \frac{2}{p}} \sigma_*^2 R^2}{\varepsilon^{2}} \right) \right\}.
        \end{equation*}
        This generalization allows one to solve problem \eqref{eq:init_problem} in a broader setting, for instance, by imposing a constraint. In particular, by solving the problem on a simplex ($p=1$, $q=\infty$) we can achieve a reduction of the total number of calls to the zero-order oracle $T$ by $\ln(d)$ compared to the Euclidean case.
    \end{remark}


\section{Experiments} \label{sec:Experiments}
In this section, we will use a simple example to verify the theoretical results, namely to show the convergence of the proposed algorithm Accelerated Zero-Order Stochastic Gradient Descent Method (AZO-SGD, see Algorithm \ref{alg:AZO_SGD}). The optimization problem \eqref{eq:init_problem} is as follows:
\begin{equation} \label{eq:experiments_problems}
    \min_{x \in \mathbb{R}^d} f(x) := \frac{1}{m} \sum_{i = 1}^{m} \left( l_{i}(x) \right)^2,
\end{equation}
where $l(x) = Ax - b$ is system of $m$ linear equations under overparameterization ($d>m$), $A \in \mathbb{R}^{m \times d}$ , $x, b \in \mathbb{R}^d$. Problem \eqref{eq:experiments_problems} is a convex stochastic optimization problem \eqref{eq:init_problem}, also known as the Empirical Risk Minimization problem, where $\xi=i$ is one of $m$ linear equations.
\vspace{-0.8cm}
\begin{figure}[H]
    \centering
    \includegraphics[width=0.5\textwidth]{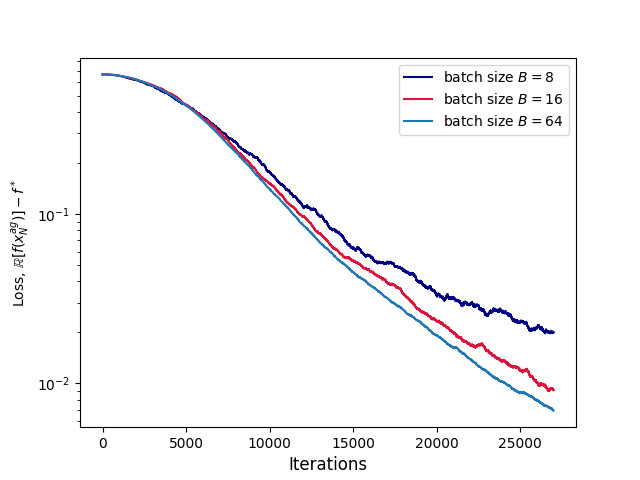}
    \caption{Convergence of the Accelerated Zero-Order Stochastic Gradient Descent Method and the effect of parameter $B$ (batch size) on the iteration complexity.}\label{Fig1}   
\end{figure}
\vspace{-0.8cm}
In Figure \ref{Fig1} we see the convergence of the proposed gradient-free algorithm. We can also conclude that as the batch size increases, the number of iterations required to achieve the desired accuracy decreases. This effect occurs especially in applications with huge-scale machine learning models. We optimize $f(x)$ \eqref{eq:experiments_problems} with parameters: $d = 256$ (dimensional of problem), $m = 128$
(number of linear equations), $\tau = 0.001$
(smoothing parameter), $\al{\gamma} = 0.0001$ (step size),
$B = \left\{ 8, 16, 64 \right\}$ (batch size). We also understand machine inaccuracy by noise level.


\section{Conclusion} \label{sec:Conclusion}
The optimization problem in the overparameterization setup has not yet been sufficiently studied. In this paper, we proposed a novel gradient-free algorithm: Accelerated Zero-Order Stochastic Gradient Descent Method for solving the smooth convex stochastic black-box optimization problem in the overparameterization setup. Our approach in creating the gradient-free Algorithm \ref{alg:AZO_SGD} was based on accelerated stochastic gradient descent. However, since there is an accumulation of adversarial noise in $l_2$ randomization, the result of \cite{Woodworth_2021_over} was generalized to the case of a biased gradient oracle. We also showed that the proposed gradient-free algorithm (AZO-SGD) is optimal in terms of iteration and oracle complexities. In addition, we obtained the first estimate, as far as we know, of the level of adversarial noise in the overparameterization setup, thereby opening up many potentially interesting future research questions in this setup.

The authors are grateful to Daniil Vostrikov.


\printbibliography

@article{Audet_2017,
  title={Derivative-free and blackbox optimization},
  author={Audet, Charles and Hare, Warren},
  year={2017},
  publisher={Springer}
}

@book{Conn_2009,
  title={Introduction to derivative-free optimization},
  author={Conn, Andrew R and Scheinberg, Katya and Vicente, Luis N},
  year={2009},
  publisher={SIAM}
}

@inproceedings{Gasnikov_2022_ICML,
  title={The power of first-order smooth optimization for black-box non-smooth problems},
  author={Gasnikov, Alexander and Novitskii, Anton and Novitskii, Vasilii and Abdukhakimov, Farshed and Kamzolov, Dmitry and Beznosikov, Aleksandr and Takac, Martin and Dvurechensky, Pavel and Gu, Bin},
  booktitle={International Conference on Machine Learning},
  pages={7241--7265},
  year={2022},
  organization={PMLR}
}

@inproceedings{Bach_2016,
  title={Highly-smooth zero-th order online optimization},
  author={Bach, Francis and Perchet, Vianney},
  booktitle={Conference on Learning Theory},
  pages={257--283},
  year={2016},
  organization={PMLR}
}

@article{Akhavan_2022,
  title={A gradient estimator via L1-randomization for online zero-order optimization with two point feedback},
  author={Akhavan, Arya and Chzhen, Evgenii and Pontil, Massimiliano and Tsybakov, Alexandre},
  journal={Advances in Neural Information Processing Systems},
  volume={35},
  pages={7685--7696},
  year={2022}
}

@article{Rosenbrock_1960,
  title={An automatic method for finding the greatest or least value of a function},
  author={Rosenbrock, HoHo},
  journal={The computer journal},
  volume={3},
  number={3},
  pages={175--184},
  year={1960},
  publisher={Oxford University Press}
}

@article{Bogolubsky_2016,
  title={Learning supervised pagerank with gradient-based and gradient-free optimization methods},
  author={Bogolubsky, Lev and Dvurechenskii, Pavel and Gasnikov, Alexander and Gusev, Gleb and Nesterov, Yurii and Raigorodskii, Andrei M and Tikhonov, Aleksey and Zhukovskii, Maksim},
  journal={Advances in neural information processing systems},
  volume={29},
  year={2016}
}

@inproceedings{Chen_2017,
  title={Zoo: Zeroth order optimization based black-box attacks to deep neural networks without training substitute models},
  author={Chen, Pin-Yu and Zhang, Huan and Sharma, Yash and Yi, Jinfeng and Hsieh, Cho-Jui},
  booktitle={Proceedings of the 10th ACM workshop on artificial intelligence and security},
  pages={15--26},
  year={2017}
}

@inproceedings{Patel_2022,
  title={Distributed online and bandit convex optimization},
  author={Patel, Kumar Kshitij and Saha, Aadirupa and Wang, Lingxiao and Srebro, Nathan},
  booktitle={OPT 2022: Optimization for Machine Learning (NeurIPS 2022 Workshop)},
  year={2022}
}

@article{Lobanov_2023_SADOM,
  title={Non-Smooth Setting of Stochastic Decentralized Convex Optimization Problem Over Time-Varying Graphs},
  author={Lobanov, Aleksandr and Konin, Georgiy and Gasnikov, Alexander and Kovalev, Dmitry},
  journal={arXiv preprint arXiv:2307.00392},
  year={2023}
}

@inproceedings{Gao_2018,
  title={Black-box generation of adversarial text sequences to evade deep learning classifiers},
  author={Gao, Ji and Lanchantin, Jack and Soffa, Mary Lou and Qi, Yanjun},
  booktitle={2018 IEEE Security and Privacy Workshops (SPW)},
  pages={50--56},
  year={2018},
  organization={IEEE}
}

@article{Scaman_2019,
  title={Optimal convergence rates for convex distributed optimization in networks},
  author={Scaman, Kevin and Bach, Francis and Bubeck, S{\'e}bastien and Lee, Yin Tat and Massouli{\'e}, Laurent},
  journal={Journal of Machine Learning Research},
  volume={20},
  pages={1--31},
  year={2019}
}

@article{Lobanov_2022,
  title={Gradient-Free Federated Learning Methods with $ l\_1 $ and $ l\_2 $-Randomization for Non-Smooth Convex Stochastic Optimization Problems},
  author={Lobanov, Aleksandr and Alashqar, Belal and Dvinskikh, Darina and Gasnikov, Alexander},
  journal={arXiv preprint arXiv:2211.10783},
  year={2022}
}

@article{Hazan_2017,
  title={Hyperparameter optimization: A spectral approach},
  author={Hazan, Elad and Klivans, Adam and Yuan, Yang},
  journal={arXiv preprint arXiv:1706.00764},
  year={2017}
}

@article{Elsken_2019,
  title={Neural architecture search: A survey},
  author={Elsken, Thomas and Metzen, Jan Hendrik and Hutter, Frank},
  journal={The Journal of Machine Learning Research},
  volume={20},
  number={1},
  pages={1997--2017},
  year={2019},
  publisher={JMLR. org}
}

@inproceedings{Papernot_2017,
  title={Practical black-box attacks against machine learning},
  author={Papernot, Nicolas and McDaniel, Patrick and Goodfellow, Ian and Jha, Somesh and Celik, Z Berkay and Swami, Ananthram},
  booktitle={Proceedings of the 2017 ACM on Asia conference on computer and communications security},
  pages={506--519},
  year={2017}
}

@article{Papernot_2016,
  title={Transferability in machine learning: from phenomena to black-box attacks using adversarial samples},
  author={Papernot, Nicolas and McDaniel, Patrick and Goodfellow, Ian},
  journal={arXiv preprint arXiv:1605.07277},
  year={2016}
}

@article{Duchi_2015,
  title={Optimal rates for zero-order convex optimization: The power of two function evaluations},
  author={Duchi, John C and Jordan, Michael I and Wainwright, Martin J and Wibisono, Andre},
  journal={IEEE Transactions on Information Theory},
  volume={61},
  number={5},
  pages={2788--2806},
  year={2015},
  publisher={IEEE}
}

@article{Shibaev_2022,
  title={Zeroth-order methods for noisy H{\"o}lder-gradient functions},
  author={Shibaev, Innokentiy and Dvurechensky, Pavel and Gasnikov, Alexander},
  journal={Optimization Letters},
  volume={16},
  number={7},
  pages={2123--2143},
  year={2022},
  publisher={Springer}
}

@article{Shamir_2017,
  title={An optimal algorithm for bandit and zero-order convex optimization with two-point feedback},
  author={Shamir, Ohad},
  journal={The Journal of Machine Learning Research},
  volume={18},
  number={1},
  pages={1703--1713},
  year={2017},
  publisher={JMLR. org}
}

@article{Bubeck_2012,
  title={Regret analysis of stochastic and nonstochastic multi-armed bandit problems},
  author={Bubeck, S{\'e}bastien and Cesa-Bianchi, Nicolo and others},
  journal={Foundations and Trends{\textregistered} in Machine Learning},
  volume={5},
  number={1},
  pages={1--122},
  year={2012},
  publisher={Now Publishers, Inc.}
}

@inproceedings{Bartlett_2008,
  title={High-probability regret bounds for bandit online linear optimization},
  author={Bartlett, Peter and Dani, Varsha and Hayes, Thomas and Kakade, Sham and Rakhlin, Alexander and Tewari, Ambuj},
  booktitle={Proceedings of the 21st Annual Conference on Learning Theory-COLT 2008},
  pages={335--342},
  year={2008},
  organization={Omnipress}
}

@inproceedings{Choromanski_2018,
  title={Structured evolution with compact architectures for scalable policy optimization},
  author={Choromanski, Krzysztof and Rowland, Mark and Sindhwani, Vikas and Turner, Richard and Weller, Adrian},
  booktitle={International Conference on Machine Learning},
  pages={970--978},
  year={2018},
  organization={PMLR}
}

@article{Mania_2018,
  title={Simple random search of static linear policies is competitive for reinforcement learning},
  author={Mania, Horia and Guy, Aurelia and Recht, Benjamin},
  journal={Advances in Neural Information Processing Systems},
  volume={31},
  year={2018}
}

@article{Gasnikov_2022,
  title={Randomized gradient-free methods in convex optimization},
  author={Gasnikov, Alexander and Dvinskikh, Darina and Dvurechensky, Pavel and Gorbunov, Eduard and Beznosikov, Aleksander and Lobanov, Alexander},
  journal={arXiv preprint arXiv:2211.13566},
  year={2022}
}

@article{Robbins_1951,
  title={A stochastic approximation method},
  author={Robbins, Herbert and Monro, Sutton},
  journal={The annals of mathematical statistics},
  pages={400--407},
  year={1951},
  publisher={JSTOR}
}

@article{Bottou_2018,
  title={Optimization methods for large-scale machine learning},
  author={Bottou, L{\'e}on and Curtis, Frank E and Nocedal, Jorge},
  journal={SIAM review},
  volume={60},
  number={2},
  pages={223--311},
  year={2018},
  publisher={SIAM}
}

@inproceedings{Dvinskikh_2022,
  title={Noisy zeroth-order optimization for non-smooth saddle point problems},
  author={Dvinskikh, Darina and Tominin, Vladislav and Tominin, Iaroslav and Gasnikov, Alexander},
  booktitle={International Conference on Mathematical Optimization Theory and Operations Research},
  pages={18--33},
  year={2022},
  organization={Springer}
}

@article{Lobanov_2023_WAIT,
  title={Stochastic Adversarial Noise in the ``Black Box" Optimization Problem},
  author={Lobanov, Aleksandr},
  journal={arXiv preprint arXiv:2304.07861},
  year={2023}
}

@article{Kornilov_2023,
  title={Gradient Free Methods for Non-Smooth Convex Optimization with Heavy Tails on Convex Compact},
  author={Kornilov, Nikita and Gasnikov, Alexander and Dvurechensky, Pavel and Dvinskikh, Darina},
  journal={arXiv preprint arXiv:2304.02442},
  year={2023}
}

@article{Lan_2012,
  title={An optimal method for stochastic composite optimization},
  author={Lan, Guanghui},
  journal={Mathematical Programming},
  volume={133},
  number={1-2},
  pages={365--397},
  year={2012},
  publisher={Springer}
}

@inproceedings{Woodworth_2021,
  title={The min-max complexity of distributed stochastic convex optimization with intermittent communication},
  author={Woodworth, Blake E and Bullins, Brian and Shamir, Ohad and Srebro, Nathan},
  booktitle={Conference on Learning Theory},
  pages={4386--4437},
  year={2021},
  organization={PMLR}
}

@article{Gorbunov_2020,
  title={Stochastic optimization with heavy-tailed noise via accelerated gradient clipping},
  author={Gorbunov, Eduard and Danilova, Marina and Gasnikov, Alexander},
  journal={Advances in Neural Information Processing Systems},
  volume={33},
  pages={15042--15053},
  year={2020}
}

@article{Gasnikov_2016,
  title={Universal fast gradient method for stochastic composit optimization problems},
  author={Gasnikov, Alexander and Nesterov, Yurii},
  journal={arXiv preprint arXiv:1604.05275},
  year={2016}
}

@article{Ajalloeian_2020,
  title={On the convergence of SGD with biased gradients},
  author={Ajalloeian, Ahmad and Stich, Sebastian U},
  journal={arXiv preprint arXiv:2008.00051},
  year={2020}
}

@article{Polyak_1963,
  title={Gradient methods for the minimisation of functionals},
  author={Polyak, Boris T},
  journal={USSR Computational Mathematics and Mathematical Physics},
  volume={3},
  number={4},
  pages={864--878},
  year={1963},
  publisher={Elsevier}
}

@article{Lojasiewicz_1963,
  title={Une propri{\'e}t{\'e} topologique des sous-ensembles analytiques r{\'e}els},
  author={Lojasiewicz, Stanislaw},
  journal={Les {\'e}quations aux d{\'e}riv{\'e}es partielles},
  volume={117},
  pages={87--89},
  year={1963}
}

@article{Yue_2022,
  title={On the Lower Bound of Minimizing Polyak-$\{$$\backslash$L$\}$ ojasiewicz functions},
  author={Yue, Pengyun and Fang, Cong and Lin, Zhouchen},
  journal={arXiv preprint arXiv:2212.13551},
  year={2022}
}

@article{Woodworth_2021_over,
  title={An even more optimal stochastic optimization algorithm: minibatching and interpolation learning},
  author={Woodworth, Blake E and Srebro, Nathan},
  journal={Advances in Neural Information Processing Systems},
  volume={34},
  pages={7333--7345},
  year={2021}
}

@inproceedings{Rakhlin_2012,
  title={Making gradient descent optimal for strongly convex stochastic optimization},
  author={Rakhlin, Alexander and Shamir, Ohad and Sridharan, Karthik},
  booktitle={Proceedings of the 29th International Coference on International Conference on Machine Learning},
  pages={1571--1578},
  year={2012}
}

@article{Hazan_2014,
  title={Beyond the regret minimization barrier: optimal algorithms for stochastic strongly-convex optimization},
  author={Hazan, Elad and Kale, Satyen},
  journal={The Journal of Machine Learning Research},
  volume={15},
  number={1},
  pages={2489--2512},
  year={2014},
  publisher={JMLR. org}
}

@article{Stich_2019,
  title={Unified optimal analysis of the (stochastic) gradient method},
  author={Stich, Sebastian U},
  journal={arXiv preprint arXiv:1907.04232},
  year={2019}
}

@book{Bertsekas_1996,
  title={Neuro-dynamic programming},
  author={Bertsekas, Dimitri and Tsitsiklis, John N},
  year={1996},
  publisher={Athena Scientific}
}

@article{Schmidt_2013,
  title={Fast convergence of stochastic gradient descent under a strong growth condition},
  author={Schmidt, Mark and Roux, Nicolas Le},
  journal={arXiv preprint arXiv:1308.6370},
  year={2013}
}

@article{Srebro_2010,
  title={Optimistic rates for learning with a smooth loss},
  author={Srebro, Nathan and Sridharan, Karthik and Tewari, Ambuj},
  journal={arXiv preprint arXiv:1009.3896},
  year={2010}
}

@article{Lobanov_2023,
  title={Highly Smoothness Zero-Order Methods for Solving Optimization Problems under PL Condition},
  author={Lobanov, Aleksandr and Gasnikov, Alexander and Stonyakin, Fedor},
  journal={arXiv preprint arXiv:2305.15828},
  year={2023}
}

@inproceedings{Vaswani_2019,
  title={Fast and faster convergence of sgd for over-parameterized models and an accelerated perceptron},
  author={Vaswani, Sharan and Bach, Francis and Schmidt, Mark},
  booktitle={The 22nd international conference on artificial intelligence and statistics},
  pages={1195--1204},
  year={2019},
  organization={PMLR}
}

@article{Zhang_2022,
  title={Adam can converge without any modification on update rules},
  author={Zhang, Yushun and Chen, Congliang and Shi, Naichen and Sun, Ruoyu and Luo, Zhi-Quan},
  journal={Advances in Neural Information Processing Systems},
  volume={35},
  pages={28386--28399},
  year={2022}
}

@article{Fatkhullin_2022,
  title={Sharp analysis of stochastic optimization under global Kurdyka-Lojasiewicz inequality},
  author={Fatkhullin, Ilyas and Etesami, Jalal and He, Niao and Kiyavash, Negar},
  journal={Advances in Neural Information Processing Systems},
  volume={35},
  pages={15836--15848},
  year={2022}
}

@inproceedings{Tran_2022,
  title={Nesterov accelerated shuffling gradient method for convex optimization},
  author={Tran, Trang H and Scheinberg, Katya and Nguyen, Lam M},
  booktitle={International Conference on Machine Learning},
  pages={21703--21732},
  year={2022},
  organization={PMLR}
}

@article{Cotter_2011,
  title={Better mini-batch algorithms via accelerated gradient methods},
  author={Cotter, Andrew and Shamir, Ohad and Srebro, Nati and Sridharan, Karthik},
  journal={Advances in neural information processing systems},
  volume={24},
  year={2011}
}

@article{Allen_2019,
  title={Learning and generalization in overparameterized neural networks, going beyond two layers},
  author={Allen-Zhu, Zeyuan and Li, Yuanzhi and Liang, Yingyu},
  journal={Advances in neural information processing systems},
  volume={32},
  year={2019}
}

@article{Belkin_2019,
  title={Reconciling modern machine-learning practice and the classical bias--variance trade-off},
  author={Belkin, Mikhail and Hsu, Daniel and Ma, Siyuan and Mandal, Soumik},
  journal={Proceedings of the National Academy of Sciences},
  volume={116},
  number={32},
  pages={15849--15854},
  year={2019},
  publisher={National Acad Sciences}
}

@article{Jacot_2018,
  title={Neural tangent kernel: Convergence and generalization in neural networks},
  author={Jacot, Arthur and Gabriel, Franck and Hongler, Cl{\'e}ment},
  journal={Advances in neural information processing systems},
  volume={31},
  year={2018}
}

@article{Vasin_2023,
  title={Accelerated gradient methods with absolute and relative noise in the gradient},
  author={Vasin, Artem and Gasnikov, Alexander and Dvurechensky, Pavel and Spokoiny, Vladimir},
  journal={Optimization Methods and Software},
  pages={1--50},
  year={2023},
  publisher={Taylor \& Francis}
}

@article{Dvinskikh_2021,
  title={Decentralized and parallel primal and dual accelerated methods for stochastic convex programming problems},
  author={Dvinskikh, Darina and Gasnikov, Alexander},
  journal={Journal of Inverse and Ill-posed Problems},
  volume={29},
  number={3},
  pages={385--405},
  year={2021}
}

@phdthesis{Devolder_2013,
  title={Exactness, inexactness and stochasticity in first-order methods for large-scale convex optimization},
  author={Devolder, Olivier},
  year={2013},
  school={CORE UCLouvain Louvain-la-Neuve, Belgium}
}

@book{Nesterov_2003,
  title={Introductory lectures on convex optimization: A basic course},
  author={Nesterov, Yurii},
  volume={87},
  year={2003},
  publisher={Springer Science \& Business Media}
}

@article{Akhavan_2023,
  title={Gradient-free optimization of highly smooth functions: improved analysis and a new algorithm},
  author={Akhavan, Arya and Chzhen, Evgenii and Pontil, Massimiliano and Tsybakov, Alexandre B},
  journal={arXiv preprint arXiv:2306.02159},
  year={2023}
}

@book{Zorich_2016,
  title={Mathematical analysis II},
  author={Zorich, Vladimir Antonovich and Paniagua, Octavio},
  volume={220},
  year={2016},
  publisher={Springer}
}

@article{Ilandarideva_2023,
  title={Accelerated stochastic approximation with state-dependent noise},
  author={Ilandarideva, Sasila and Juditsky, Anatoli and Lan, Guanghui and Li, Tianjiao},
  journal={arXiv preprint arXiv:2307.01497},
  year={2023}
}

@book{Pardalos_2021,
  title={Black Box Optimization, Machine Learning, and No-Free Lunch Theorems},
  author={Pardalos, Panos M and Rasskazova, Varvara and Vrahatis, Michael N and others},
  year={2021},
  publisher={Springer}
}

\newpage

\appendix
\begin{center}
        \LARGE \bf APPENDIX
    \end{center}

\section{Auxiliary Facts and Results}

    In this section we list auxiliary facts and results that we use several times in our~proofs.
    
    \subsection{Squared norm of the sum} For all $a_1,...,a_n \in \mathbb{R}^d$, where $n=\{2,3\}$
    \begin{equation}
        \label{eq:squared_norm_sum}
        \norms{a_1 + ... + a_n }^2 \leq n \norms{ a_1 }^2 + ... + n \norms{a_n}^2.
    \end{equation}

    \subsection{$L$ smoothness function}
        Function $f$ is called $L$-smooth on $\mathbb{R}^d$ with $L~>~0$ when it is differentiable and its gradient is $L$-Lipschitz continuous on $\mathbb{R}^d$, i.e.\ 
        \begin{equation}
            \norms{\nabla f(x) - \nabla f(y)} \leq L \norms{x - y},\quad \forall x,y\in \mathbb{R}^d. \label{eq:L_smoothness}
        \end{equation}
         It is well-known that $L$-smoothness implies (see e.g., \cite{Nesterov_2003})
        \begin{eqnarray*}
            f(y) \leq f(x) + \dotprod{\nabla f(x)}{y-x} + \frac{L}{2}\norms{y-x}^2\quad \forall x,y\in \mathbb{R}^d,
        \end{eqnarray*} 
        and if $f$ is additionally convex, then
        \begin{eqnarray*}
            \norms{ \nabla f(x) - \nabla f(y) }^2 \leq 2L \left( f(x) - f(y) - \dotprod{ \nabla f(y)}{x-y} \right) \quad \forall x,y \in \mathbb{R}^d. 
        \end{eqnarray*}

    \subsection{Wirtinger-Poincare inequality}
        Let $f$ is differentiable, then for all $x \in \mathbb{R}^d$, $\tau e \in S^d_2(\tau)$:
        \begin{equation}\label{eq:Wirtinger_Poincare}
            \expect{f(x+ \tau e)^2} \leq \frac{\tau^2}{d} \expect{\norms{\nabla f(x + \tau e)}^2}.
        \end{equation}

    \newpage
\section{Proof Theorem \ref{th:biased_AC_SA}} \label{Appendix:proof_th1}
    In this section, our reasoning will be based on the proof from \cite{Woodworth_2021_over}. Initially, let us formally define a batched biased gradient oracle (see Definition \ref{def:biased_oracle}):
    \begin{equation}
        \label{eq:batched_gradient_oracle}
        \gg^B(x) := \frac{1}{B} \sum_{i=1}^{B} \gg(x, \xi_i), \quad \text{for i.i.d. } \xi_1, \xi_2, ..., \xi_B \sim \mathcal{D}. 
    \end{equation}
    Then Algorithm \ref{alg:AC_SA} presents a Biased Accelerated Mini-batch Stochastic Gradient Descent (Biased AC-SA method) under the overparameterization condition. 
    \begin{algorithm}
        \caption{Biased AC-SA}\label{alg:AC_SA}
        \textbf{Input}: Start point $x^{ag}_0 = x_0 \in \mathbb{R}^d$, maximum number of iterations $N \in \mathbb{Z}_+$.\\
        \hspace*{\algorithmicindent} Let stepsize $ \al{\gamma}_k > 0$, parameters $\beta_k, \gamma > 0$, batch size $B \in \mathbb{Z}_+$. 
        \begin{algorithmic}[1]
        \For{$k=0,...,N-1$}
        \State $\beta_k = 1 + \frac{k}{6}$ and $\al{\gamma}_k = \al{\gamma} (k+1)$ for $\gamma = \min \left\{ \frac{1}{12 L}, \frac{B}{24 L(N+1)}, \sqrt{\frac{B R^2}{Lf^* N^3}} \right\}$  
        \State $x^{md}_k = \beta^{-1}_k x_k + (1 - \beta^{-1}_k) x_k^{ag}$
        \State $\tilde{x}_{k+1} = x_k - \al{\gamma}_k \gg_k^B(x_k^{md})$, where $\gg_k^B(x_k^{md})$ is defined from \eqref{eq:batched_gradient_oracle}
        \State $x_{k+1} = \min \left\{ 1, \frac{R}{\norms{\tilde{x}_{k+1}}} \right\} \tilde{x}_{k+1}$
        \State $x_{k+1}^{ag} = \beta^{-1}_k x_{k+1} + (1 - \beta^{-1}_k) x_{k}^{ag}$
        \EndFor
        \end{algorithmic}
        \textbf{Output}: $x_N^{ag}$.
    \end{algorithm}
    
    Then it is not hard to show that the following Lemma is also correct for the batched biased gradient oracle \eqref{eq:batched_gradient_oracle}. Therefore, to avoid repetition, we formulate this lemma without proof, by referring to the original proof. 
    \begin{lemma}[see Lemma 1, \cite{Woodworth_2021_over}]\label{lemma:lem1}
    Let $x_{k+1}, x_k$ and $x_k^{md}$ be updated as in Algorithm \ref{alg:AC_SA}. Then for any $x \in \left\{ x: \norms{x}\leq R \right\}$
        \begin{eqnarray*}
            \gamma_k \dotprod{\gg^B(x^{md}_k)}{x_{k+1} - x_{k}^{md}} &\leq& \gamma_k \dotprod{\gg^B(x^{md}_k)}{x - x_{k}^{md}} + \frac{1}{2} \norms{x - x_k}^2 \\
            && \quad - \frac{1}{2} \norms{x - x_{k+1}}^2 - \frac{1}{2} \norms{x_{k+1} - x_k}^2.
        \end{eqnarray*} 
    \end{lemma}
    Next, we provide some auxiliary lemma before presenting proof of Theorem~\ref{th:biased_AC_SA}.
    \begin{lemma}\label{lemma:lem2}
        Let function $f(x,\xi)$ satisfy Assumptions \ref{ass:convex}-\ref{ass:smooth} and the gradient oracle (see Definition \ref{def:biased_oracle}) satisfy Assumption \ref{ass:stoch_noise}, and let $\gg^B(x^{md}_k)$ be defined in \eqref{eq:batched_gradient_oracle}. Then
        \begin{eqnarray*}
            \expect{\norms{\gg^B(x_k^{md}) - \nabla f(x_k^{md})}^2} \leq \frac{8 L^2 R^2}{B \beta_k^2} + \frac{8 L}{B} \expect{f(x_k^{ag}) - f^*} + \frac{4 \sigma^2_{*}}{B} + \norms{\bb(x_k^{md})}^2
        \end{eqnarray*}
    \end{lemma}
    \begin{proof}
        \begin{align}
            \mathbb{E}\|&\gg^B(x_k^{md}) - \nabla f(x_k^{md})\|^2 \nonumber \\
            &= \expect{\norms{\frac{1}{B} \sum_{i=1}^{B} \nabla f(x^{md}_k, \xi_i) - \nabla f(x_k^{md})}^2} 
            \nonumber \\ 
            & \quad \quad + \expect{\norms{\gg^B(x_k^{md}) - \frac{1}{B} \sum_{i=1}^{B} \nabla f(x^{md}_k, \xi_i)}^2}
            \nonumber \\ 
            &\overset{\eqref{eq:batched_gradient_oracle}}{=} \expect{\norms{\frac{1}{B} \sum_{i=1}^{B} \nabla f(x^{md}_k, \xi_i) - \nabla f(x_k^{md})}^2}  + \expect{\norms{\frac{1}{B} \sum_{i=1}^{B} \bb(x_k^{md})}^2}
            \nonumber \\ 
            &= \frac{1}{B^2} \sum_{i=1}^{B} \expect{\norms{ \nabla f(x^{md}_k, \xi_i) - \nabla f(x_k^{md})}^2} + \norms{\bb(x_k^{md})}^2
            \nonumber \\ 
            &\leq \frac{1}{B} \expect{\norms{\nabla f(x_k^{md}, \xi_1)}^2} + \norms{\bb(x_k^{md})}^2
            \nonumber \\ 
            &\overset{\eqref{eq:squared_norm_sum}}{\leq} \frac{2}{B} \expect{\norms{\nabla f(x_k^{md}, \xi_1) - \nabla f(x_k^{ag}, \xi_1)}^2}  + \frac{2}{B} \expect{\norms{\nabla f(x_k^{ag}, \xi_1)}^2} + \norms{\bb(x_k^{md})}^2
            \nonumber \\ 
            & \overset{\eqref{eq:L_smoothness}}{\leq} \frac{2L^2}{B} \expect{\norms{x_k^{md} - x_k^{ag}}^2} + \frac{4}{B} \expect{\norms{\nabla f(x_k^{ag}, \xi_1) - \nabla f(x^*, \xi_1)}^2}
            \nonumber \\ 
            & \quad \quad + \frac{4}{B} \expect{\norms{\nabla f(x^*, \xi_1)}^2} + \norms{\bb(x_k^{md})}^2.
            \label{eq:lemma2}
        \end{align}
        For the first term on the right hand side:
        \begin{equation}
            \label{eq:lemma2_1}
            x^{md}_k = \beta^{-1}_k x_k + (1 - \beta^{-1}_k) x^{ag}_k  \Rightarrow  \norms{x_k^{md} - x^{ag}_k} = \beta_k^{-1} \norms{x_k - x_k^{ag}} \leq 2 R \beta^{-1}_k.
        \end{equation}
        For second term, we apply Theorem 2.1.5 from \cite{Nesterov_2003}:
        \begin{align}
            \mathbb{E}\|\nabla f(x_k^{ag}, \xi_1) &- \nabla f(x^*, \xi_1) \|^2
            \nonumber \\
            & \leq 2L \expect{f(x_k^{ag}, \xi_1) - f(x^*, \xi_1) - \dotprod{\nabla f(x^*, \xi_1)}{x_k^{ag} - x^*}}
            \nonumber \\
            & = 2 L \expect{f(x^{ag}_k) - f^*}. \label{eq:lemma2_2}
        \end{align}
        For third term, we apply Assumption \ref{ass:stoch_noise}:
        \begin{equation}\label{eq:lemma2_3}
            \expect{\norms{\nabla f(x^*, \xi_1)}^2} = \expect{\norms{\nabla f(x^*, \xi_1) - \nabla f(x^*)}^2} \leq \sigma_*^2.
        \end{equation}
        Substituting \eqref{eq:lemma2_1}-\eqref{eq:lemma2_3} into \eqref{eq:lemma2} we complete the proof of the Lemma. 
        \\ \qed
    \end{proof}
    We can now proceed to prove the main theorem of Section \ref{sec:biased_gradient}.\\
    \begin{flushleft}
        \textit{ Proof of the Theorem \ref{th:biased_AC_SA}}. 
    \end{flushleft}
    Using the convexity and $L$-smoothness of the function $f$ we can obtain the following upper bound:
    \begin{align}
        \beta_k \gamma_k f(x_{k+1}^{ag}) &\leq \beta_{k} \gamma_{k} \left[ f(x_k^{md}) + \dotprod{\nabla f(x_{k}^{md})}{x_{k+1}^{ag} - x_{k}^{md}} + \frac{L}{2} \norms{x_{k+1}^{ag} - x_{k}^{md}}^2 \right]
        \nonumber \\
        & = \beta_{k} \gamma_{k} \left[ f(x_k^{md}) + \dotprod{\nabla f(x_{k}^{md})}{x_{k+1}^{ag} - x_{k}^{md}}\right] + \frac{L \gamma_k}{2 \beta_k} \norms{x_{k+1} - x_{k}}^2 
        \nonumber \\
        & = \beta_{k} \gamma_{k} \left[ f(x_k^{md}) + \dotprod{\nabla f(x_{k}^{md})}{\beta^{-1}_k x_{k+1} + (1-\beta^{-1})x^{ag}_k - x_{k}^{md}}\right] 
        \nonumber \\
        & \quad \quad + \frac{L \gamma_k}{2 \beta_k} \norms{x_{k+1} - x_{k}}^2 
        \nonumber \\
        & = (\beta_{k} - 1) \gamma_{k} \left[ f(x_k^{md}) + \dotprod{\nabla f(x_{k}^{md})}{x_k^{ag} - x_{k}^{md}}\right] 
        \nonumber \\
        & \quad \quad + \gamma_{k} \left[ f(x_k^{md}) + \dotprod{\nabla f(x_{k}^{md})}{x_{k+1} - x_{k}^{md}}\right]  + \frac{L \gamma_k}{2 \beta_k} \norms{x_{k+1} - x_{k}}^2 
        \nonumber \\
        & \leq (\beta_k - 1) \gamma_k f(x_k^{ag}) + \gamma_k \left[ f(x_k^{md}) + \dotprod{\gg^B(x_k^{md})}{x_{k+1} - x_k^{md}} \right] 
        \nonumber \\
        & \quad \quad - \gamma_k \dotprod{\gg^B(x_k^{md}) - \nabla f(x_k^{md})}{x_{k+1} - x_{k}^{md}} + \frac{L \gamma_k}{2 \beta_k} \norms{x_{k+1} - x_{k}}^2
        \nonumber
    \end{align}
    Using Lemma \ref{lemma:lem1} with $x = x^* \in \argmin_{x: \norms{x} \leq R} f(x)$ for second term we obtain:
    \begin{align}
        \gamma_k  f(x_k^{md}) &+ \gamma_k  \dotprod{\gg^B(x_k^{md})}{x_{k+1} - x_k^{md}}
        \nonumber \\
        & = \gamma_k  f(x_k^{md}) + \gamma_k  \dotprod{\gg^B(x_k^{md})}{x^* - x_k^{md}}
        \nonumber \\
        & \quad \quad + \frac{1}{2} \norms{x^*-x_k}^2 - \frac{1}{2} \norms{x^*-x_{k+1}}^2 - \frac{1}{2} \norms{x_{k+1}-x_k}^2
        \nonumber \\
        & = \gamma_k  f(x_k^{md}) + \gamma_k  \dotprod{\nabla f(x_k^{md})}{x^* - x_k^{md}} + \gamma_k  \dotprod{\gg^B(x_k^{md}) - \nabla f(x_k^{md})}{x^* - x_k^{md}}
         \nonumber \\
        & \quad \quad + \frac{1}{2} \norms{x^*-x_k}^2 - \frac{1}{2} \norms{x^*-x_{k+1}}^2 - \frac{1}{2} \norms{x_{k+1}-x_k}^2
        \nonumber \\
        & \leq \gamma_k f^* + \gamma_k  \dotprod{\gg^B(x_k^{md}) - \nabla f(x_k^{md})}{x^* - x_k^{md}}
        \nonumber \\
        & \quad \quad + \frac{1}{2} \norms{x^*-x_k}^2 - \frac{1}{2} \norms{x^*-x_{k+1}}^2 - \frac{1}{2} \norms{x_{k+1}-x_k}^2.
        \nonumber
    \end{align}
    Substituting the obtained upper bound we obtain:
\allowdisplaybreaks
    \begin{align}
        \beta_k \gamma_k f(x_{k+1}^{ag}) &\leq (\beta_k - 1) \gamma_k f(x_k^{ag}) + \gamma_k f^* + \frac{L \gamma_k}{2 \beta_k} \norms{x_{k+1} - x_{k}}^2
        \nonumber \\
        & \quad \quad + \gamma_k \dotprod{\gg^B(x_k^{md}) - \nabla f(x_k^{md})}{x^* - x_{k+1}}
        \nonumber \\
        & \quad \quad + \frac{1}{2} \left( \norms{x^*-x_k}^2 - \norms{x^*-x_{k+1}}^2 - \norms{x_{k+1}-x_k}^2 \right) .
        \nonumber
    \end{align}
    Adding $\beta_k \gamma_k f^*$ to both sides we can obtain:
    \begin{align}
        \beta_k \gamma_k \left[ f(x_{k+1}^{ag}) - f^* \right] &\leq (\beta_k - 1) \gamma_k \left[ f(x_k^{ag}) - f^* \right] +\frac{1}{2} \norms{x_k - x^*}^2 - \frac{1}{2} \norms{x_{k+1} - x^*}^2
        \nonumber \\
        & \quad \quad + \gamma_k \dotprod{\gg^B(x_k^{md}) - \nabla f(x_k^{md})}{x^* - x_{k+1}}
        \nonumber \\
        & \quad \quad + \frac{L \gamma_k - \beta_k}{2 \beta_k} \norms{x_{k} - x_{k+1}}^2
        \nonumber \\
        & = (\beta_k - 1) \gamma_k \left[ f(x_k^{ag}) - f^* \right] +\frac{1}{2} \norms{x_k - x^*}^2 - \frac{1}{2} \norms{x_{k+1} - x^*}^2
        \nonumber \\
        & \quad \quad + \gamma_k \dotprod{\gg^B(x_k^{md}) - \nabla f(x_k^{md})}{x^* - x_{k}}
        \nonumber \\
        & \quad \quad + \gamma_k \dotprod{\gg^B(x_k^{md}) - \nabla f(x_k^{md})}{x_{k} - x_{k+1}}
        \nonumber \\
        & \quad \quad + \frac{L \gamma_k - \beta_k}{2 \beta_k} \norms{x_{k} - x_{k+1}}^2 
        \nonumber \\
        & \leq (\beta_k - 1) \gamma_k \left[ f(x_k^{ag}) - f^* \right] +\frac{1}{2} \norms{x_k - x^*}^2 - \frac{1}{2} \norms{x_{k+1} - x^*}^2
        \nonumber \\
        & \quad \quad + \gamma_k \dotprod{\gg^B(x_k^{md}) - \nabla f(x_k^{md})}{x^* - x_{k}}
        \nonumber \\
        & \quad \quad + \gamma_k \norms{\gg^B(x_k^{md}) - \nabla f(x_k^{md})} \norms{x_{k} - x_{k+1}}
        \nonumber \\
        & \quad \quad + \frac{L \gamma_k - \beta_k}{2 \beta_k} \norms{x_{k} - x_{k+1}}^2. 
        \nonumber
    \end{align}
    Since $\beta_k = 1 + \frac{k}{6} > \frac{1+k}{6} \geq 2L\gamma_k$, then 
    \begin{equation*}
        \gamma_k \norms{\gg^B(x_k^{md}) - \nabla f(x_k^{md})} \norms{x_{k} - x_{k+1}} + \frac{L \gamma_k - \beta_k}{2 \beta_k} \norms{x_{k} - x_{k+1}}^2
    \end{equation*} 
    is a quadratic polynomial of the form: $-\frac{a}{2}y^2 + by$ (where $y = \norms{x_k - x_{k+1}}$), which can be upper bounded by $-\frac{a}{2}y^2 + b y\leq \max_{y} \left\{ -\frac{a}{2}y^2 + b y \right\} = \frac{b^2}{2a}$. We get
    \begin{align}
        \beta_k \gamma_k \left[ f(x_{k+1}^{ag}) - f^* \right] &\leq (\beta_k - 1) \gamma_k \left[ f(x_k^{ag}) - f^* \right] +\frac{1}{2} \norms{x_k - x^*}^2 - \frac{1}{2} \norms{x_{k+1} - x^*}^2
        \nonumber \\
        & \quad \quad + \gamma_k \dotprod{\gg^B(x_k^{md}) - \nabla f(x_k^{md})}{x^* - x_{k}}
        \nonumber \\
        & \quad \quad + \frac{\beta_k \gamma_k^2}{2 (\beta_k - L \gamma_k)} \norms{\gg^B(x_k^{md}) - \nabla f(x_k^{md})}^2
        \nonumber \\ 
        & \leq (\beta_k - 1) \gamma_k \left[ f(x_k^{ag}) - f^* \right] +\frac{1}{2} \norms{x_k - x^*}^2 - \frac{1}{2} \norms{x_{k+1} - x^*}^2
        \nonumber \\
        & \quad \quad + \gamma_k \dotprod{\gg^B(x_k^{md}) - \nabla f(x_k^{md})}{x^* - x_{k}}
        \nonumber \\
        & \quad \quad + \gamma_k^2 \norms{\gg^B(x_k^{md}) - \nabla f(x_k^{md})}^2
        \nonumber \\
         & \leq (\beta_k - 1) \gamma_k \left[ f(x_k^{ag}) - f^* \right] +\frac{1}{2} \norms{x_k - x^*}^2 - \frac{1}{2} \norms{x_{k+1} - x^*}^2
        \nonumber \\
        & \quad \quad + \gamma_k \dotprod{\frac{1}{B}\sum_{i=1}^{B} \nabla f(x_{k}^{md}, \xi_i) - \nabla f(x_k^{md})}{x^* - x_{k}}
        \nonumber \\
        & \quad \quad + \gamma_k \dotprod{\gg^B(x_k^{md}) - \frac{1}{B}\sum_{i=1}^{B} \nabla f(x_{k}^{md}, \xi_i)}{x^* - x_{k}}
        \nonumber \\
        & \quad \quad + \gamma_k^2 \norms{\gg^B(x_k^{md}) - \nabla f(x_k^{md})}^2
        \nonumber \\
        & \leq (\beta_k - 1) \gamma_k \left[ f(x_k^{ag}) - f^* \right] +\frac{1}{2} \norms{x_k - x^*}^2 - \frac{1}{2} \norms{x_{k+1} - x^*}^2
        \nonumber \\
        & \quad \quad + \gamma_k \dotprod{\frac{1}{B}\sum_{i=1}^{B} \nabla f(x_{k}^{md}, \xi_i) - \nabla f(x_k^{md})}{x^* - x_{k}}
        \nonumber \\
        & \quad \quad + \gamma_k^2 \norms{\gg^B(x_k^{md}) - \nabla f(x_k^{md})}^2 + \gamma_k \dotprod{\bb(x_{k}^{md})}{x^* - x_{k}}.
        \nonumber
    \end{align}
    Taking the expectation of both sides we have:
    \begin{align}
        \beta_k \gamma_k  \expect{f(x_{k+1}^{ag}) - f^*} & \leq (\beta_k - 1) \gamma_k \expect{ f(x_k^{ag}) - f^*} +\frac{1}{2} \expect{\norms{x_k - x^*}^2} 
        \nonumber \\
        & \quad \quad - \frac{1}{2} \expect{\norms{x_{k+1} - x^*}^2} + \gamma_k^2 \expect{\norms{\gg^B(x_k^{md}) - \nabla f(x_k^{md})}^2} 
        \nonumber \\
        & \quad \quad + \gamma_k \dotprod{\bb(x_{k}^{md})}{x^* - x_{k}}.
        \nonumber
    \end{align}
    Using the Lemma \ref{lemma:lem2} we obtain:
    \begin{align}
        \beta_k \gamma_k  \expect{f(x_{k+1}^{ag}) - f^*} & \leq (\beta_k - 1) \gamma_k \expect{ f(x_k^{ag}) - f^*} +\frac{1}{2} \expect{\norms{x_k - x^*}^2} 
        \nonumber \\
        & \quad \quad - \frac{1}{2} \expect{\norms{x_{k+1} - x^*}^2} +  \frac{8 L^2 R^2 \gamma_k^2}{B \beta_k^2}
        \nonumber \\
        & \quad \quad  + \frac{8 L \gamma_k^2}{B} \expect{f(x_k^{ag}) - f^*} + \frac{4 \sigma^2_{*} \gamma_k^2}{B} 
        \nonumber \\
        & \quad \quad + \gamma_k \dotprod{\bb(x_{k}^{md})}{x^* - x_{k}} + \gamma_k^2 \norms{\bb(x_k^{md})}^2
        \nonumber \\
        & \leq \left(\beta_k - 1 + \frac{8 L \gamma_k}{B} \right) \gamma_k \expect{ f(x_k^{ag}) - f^*} 
        \nonumber \\
        & \quad \quad +\frac{1}{2} \expect{\norms{x_k - x^*}^2}  - \frac{1}{2} \expect{\norms{x_{k+1} - x^*}^2} 
        \nonumber \\
        & \quad \quad + \frac{8 L^2 R^2 \gamma_k^2}{B \beta_k^2} + \frac{4 \sigma^2_{*} \gamma_k^2}{B} 
        \nonumber \\
        & \quad \quad + \gamma_k \dotprod{\bb(x_{k}^{md})}{x^* - x_{k}} + \gamma_k^2 \norms{\bb(x_k^{md})}^2.
        \nonumber
    \end{align}
    We now remind that
    \begin{align}
        \beta_k &= 1 + \frac{k}{6};
        \nonumber \\
        \gamma_k &= \gamma(k+1);
        \nonumber \\
        \gamma &\leq \min\left\{ \frac{1}{12 L}, \frac{B}{24 L (N+1)} \right\} .
        \nonumber
    \end{align}
    This ensure that $\forall k:$ $\beta_k \geq 1$ and $2L \gamma_k \leq \beta_k$. Moreover, for $k \in [0; N-1]$:
    \begin{align}
        \left(\beta_{k+1} - 1 + \frac{8 L \gamma_{k+1}}{B} \right) &\gamma_{k+1} - \beta_k \gamma_k
        \nonumber \\
        & = \left( \beta_k - \frac{5}{6} + \frac{8 L \gamma_{k+1}}{B} \right) \gamma(k+2) - \beta_k \gamma(k+1)
        \nonumber \\
        & = \gamma \left( 1 + \frac{k}{6} - \frac{5 (k+2)}{6} + \frac{8 L \gamma (k+2)^2}{B} \right)
        \nonumber \\
        & = \gamma \left( - \frac{2}{3} - \frac{2k}{3} + \frac{(k+2)}{3} \cdot \frac{24 L (k+2) \gamma}{B} \right)
        \nonumber \\
        & \leq \gamma \left( -\frac{k}{3} \right) \leq 0.
        \nonumber
    \end{align}
    Thus we have shown that for $k \in [0; N-1]$: $\left(\beta_{k+1} - 1 + \frac{8 L \gamma_{k+1}}{B} \right) \gamma_{k+1} \leq \beta_k \gamma_k$. Therefore, we can conclude the following:
    \begin{align}
        \beta_k \gamma_k  \expect{f(x_{k+1}^{ag}) - f^*} & \leq \left(\beta_k - 1 + \frac{8 L \gamma_k}{B} \right) \gamma_k \expect{ f(x_k^{ag}) - f^*} 
        \nonumber \\
        & \quad \quad +\frac{1}{2} \expect{\norms{x_k - x^*}^2}  - \frac{1}{2} \expect{\norms{x_{k+1} - x^*}^2} 
        \nonumber \\
        & \quad \quad + \frac{8 L^2 R^2 \gamma_k^2}{B \beta_k^2} + \frac{4 \sigma^2_{*} \gamma_k^2}{B} 
        \nonumber \\
        & \quad \quad + \gamma_k \dotprod{\bb(x_{k}^{md})}{x^* - x_{k}} + \gamma_k^2 \norms{\bb(x_k^{md})}^2
        \nonumber \\
        & \leq \beta_{k-1} \gamma_{k-1} \expect{ f(x_k^{ag}) - f^*} 
        \nonumber \\
        & \quad \quad +\frac{1}{2} \expect{\norms{x_k - x^*}^2}  - \frac{1}{2} \expect{\norms{x_{k+1} - x^*}^2} 
        \nonumber \\
        & \quad \quad + \frac{8 L^2 R^2 \gamma_k^2}{B \beta_k^2} + \frac{4 \sigma^2_{*} \gamma_k^2}{B} 
        \nonumber \\
        & \quad \quad + \gamma_k \dotprod{\bb(x_{k}^{md})}{x^* - x_{k}} + \gamma_k^2 \norms{\bb(x_k^{md})}^2.
        \nonumber
    \end{align}
    Summing the both sides over $k$ we obtain:
    \begin{align}
        \sum_{k = 0}^{N-1} \beta_k \gamma_k  \expect{f(x_{k+1}^{ag}) - f^*} 
        & \leq \sum_{k = 0}^{N-1} \beta_{k-1} \gamma_{k-1} \expect{ f(x_k^{ag}) - f^*} 
        \nonumber \\
        & \quad \quad + \sum_{k = 0}^{N-1} \left(\frac{1}{2} \expect{\norms{x_k - x^*}^2}  - \frac{1}{2} \expect{\norms{x_{k+1} - x^*}^2} \right)
        \nonumber \\
        & \quad \quad + \sum_{k = 0}^{N-1} \frac{8 L^2 R^2 \gamma_k^2}{B \beta_k^2} + \sum_{k = 0}^{N-1} \frac{4 \sigma^2_{*} \gamma_k^2}{B} 
        \nonumber \\
        & \quad \quad + \sum_{k = 0}^{N-1} \gamma_k \dotprod{\bb(x_{k}^{md})}{x^* - x_{k}} + \sum_{k = 0}^{N-1} \gamma_k^2 \norms{\bb(x_k^{md})}^2.
        \nonumber
    \end{align}
    Simplifying the expression we have
    \begin{align}
        \beta_{N-1} \gamma_{N-1}  \expect{f(x_{N}^{ag}) - f^*} 
        & \leq \frac{1}{2} \expect{\norms{x_0 - x^*}^2}  + \sum_{k = 0}^{N-1} \frac{288 L^2 R^2 \gamma^2(k+1)^2}{B (k+6)^2} + \sum_{k = 0}^{N-1} \frac{4 \sigma^2_{*} \gamma^2(k+1)^2}{B} 
        \nonumber \\
        & \quad \quad + \sum_{k = 0}^{N-1} \gamma (k+1) \dotprod{\bb(x_{k}^{md})}{x^* - x_{k}} + \sum_{k = 0}^{N-1} \gamma^2(k+1)^2 \norms{\bb(x_k^{md})}^2
        \nonumber \\
        & \leq \frac{R^2}{2}  + \frac{288 L^2 R^2 \gamma^2 N}{B} + \frac{12 \sigma^2_{*} \gamma^2 N^3}{B} + 2 \gamma N^2 R \zeta + \gamma^2 N^3 \zeta^2,
        \nonumber 
    \end{align}
    where $\norms{\bb(x_k^{md})}^2 \leq \zeta^2$. Divide the left and right side by $\beta_{N-1} \gamma_{N-1} \simeq \gamma N^2$:
    \begin{align}
        \expect{f(x_{N}^{ag}) - f^*} 
        & \leq \frac{R^2}{2 \gamma N^2}  + \frac{288 L^2 R^2 \gamma}{B N} + \frac{12 \sigma^2_{*} \gamma N}{B} + 2 R \zeta + \gamma N \zeta^2.
        \nonumber 
    \end{align}
    With our choice of $\gamma = \min \left\{ \frac{1}{12 L}, \frac{B}{24 L (N+1)}, \sqrt{\frac{BR^2}{\sigma_*^2 N^3}} \right\}$ we obtain:
    \begin{align}
        \expect{f(x_{N}^{ag}) - f^*} 
        & \lesssim \frac{L R^2}{ N^2}  + \frac{L R^2 }{BN} + \frac{\sigma_{*}R}{\sqrt{BN}} +  \zeta R + \frac{\zeta^2 N}{2L}.
        \nonumber 
    \end{align}
    \qed

    \section{Proof Theorem on the convergence of AZO-SGD} \label{Appendix:proof_th2}
    In this section, we present a detailed proof of the results of Theorem \ref{th:AZO_SGD}. First, we find the bias and the second moment of the gradient approximation \eqref{eq:gradient_approximation} based on the improved analysis of the paper \cite{Akhavan_2023}.
    \paragraph{Bias of gradient approximation}
    Using the variational representation of the Euclidean norm, and definition of gradient approximation \eqref{eq:gradient_approximation} we can write:
    \begin{align}
        \norms{\expect{\gg(x_k,\xi,e)} - \nabla f(x_k)} &= \norms{\expect{\frac{d}{2 \tau}\left( f_\delta(x_k + \tau e, \xi) - f_\delta(x_k - \tau e, \xi) \right) e} - \nabla f(x_k)}
        \nonumber \\
        & \overset{\circledOne}{=} \norms{\expect{\frac{d}{\tau}\left( f(x_k + \tau e, \xi) + \delta(x_k + \tau e) \right) e} - \nabla f(x_k)}
        \nonumber \\
        & \overset{\circledTwo}{\leq}  \norms{\expect{\frac{d}{\tau} f(x_k + \tau e, \xi) e} - \nabla f(x_k)} + \frac{d \Delta}{\tau}
        \nonumber \\
        & \overset{\circledThree}{=}   \norms{\expect{\nabla f(x_k + \tau u, \xi)} - \nabla f(x_k)} + \frac{d \Delta}{\tau}
        \nonumber \\
        & = \sup_{z \in S_2^d(1)} \expect{| \nabla_z f(x_k + \tau u, \xi) - \nabla_z f(x_k)|} + \frac{d \Delta}{\tau}
        \nonumber \\
        & \overset{\eqref{eq:L_smoothness}}{\leq} L \tau \expect{\norms{u}} + \frac{d \Delta}{\tau} 
        \nonumber \\
        & \leq L \tau + \frac{d \Delta}{\tau}, \label{eq:proof_bias}
    \end{align}
    where $u \in B_2^d(1)$, $\circledOne =$ the equality is obtained from the fact, namely, distribution of $e$ is symmetric, $\circledTwo =$ the inequality is obtain from bounded noise $|\delta(x)| \leq \Delta$, $\circledThree =$ the equality is obtained from a version of Stokes’ theorem \cite{Zorich_2016}.

    \paragraph{Bounding second moment of gradient approximation} By definition gradient approximation \eqref{eq:gradient_approximation} and Wirtinger-Poincare inequality \eqref{eq:Wirtinger_Poincare} we have
    \begin{align}
        \expect{\norms{\gg(x^*,\xi,e)}^2} & = \frac{d^2}{4 \tau^2} \expect{\norms{\left(f_\delta(x^* + \tau e, \xi) - f_\delta(x^* - \tau e, \xi)\right) e}^2}
        \nonumber \\
        & = \frac{d^2}{4 \tau^2} \expect{\left(f(x^* + \tau e, \xi) - f(x^* - \tau e, \xi) + \delta (x^* + \tau e) - \delta (x^* -\tau e)\right)^2}
        \nonumber \\
        & \overset{\eqref{eq:squared_norm_sum}}{\leq} \frac{d^2}{2 \tau^2} \left( \expect{\left(f(x^* + \tau e, \xi) - f(x^* - \tau e, \xi)\right)^2} + 2 \Delta^2 \right)
        \nonumber \\
        & \overset{\eqref{eq:Wirtinger_Poincare}}{\leq} \frac{d^2}{2 \tau^2} \left( \frac{\tau^2}{d} \expect{\norms{ \nabla f(x^* + \tau e, \xi) + \nabla f(x^* - \tau e, \xi)}^2} + 2 \Delta^2 \right) 
        \nonumber \\
        & = \frac{d^2}{2 \tau^2} \left( \frac{\tau^2}{d} \expect{\norms{ \nabla f(x^* + \tau e, \xi) + \nabla f(x^* - \tau e, \xi) \pm 2 \nabla f(x^*, \xi)}^2} + 2 \Delta^2 \right)
        \nonumber \\
        & \overset{\eqref{eq:L_smoothness}}{\leq} 4d \norms{\nabla f(x^*, \xi)}^2 + 4 d  L^2 \tau^2 \expect{\norms{e}^2}  + \frac{d^2 \Delta^2}{\tau^2}  
        \nonumber \\
        & \overset{\circledOne}{\leq} 4d \sigma^2_* + 4 d  L^2 \tau^2 \expect{\norms{e}^2}  + \frac{d^2 \Delta^2}{\tau^2}, \label{eq:proof_variance}
    \end{align}
    where $\circledOne =$ the inequality is obtain from Assumption \ref{ass:stoch_noise}.

   We can now explicitly write down the convergence of the gradient-free AZO-SGD method (see Section \ref{sec:Main_Result}, Algorithm \ref{alg:AZO_SGD}) by substituting upper bounds on the bias \eqref{eq:proof_bias} and second moment \eqref{eq:proof_variance} for the gradient approximation \eqref{eq:gradient_approximation} in the convergence of the first-order method: Biased AC-SA (see Theorem \ref{th:biased_AC_SA}):
   \begin{align}
       \expect{f(x_{N}^{ag}) - f^*} 
        & \lesssim \underbrace{\frac{L R^2}{ N^2}}_{\circledOne}  + \underbrace{\frac{L R^2 }{BN}}_{\circledTwo} + \underbrace{\frac{\sqrt{d} \sigma_{*}R}{\sqrt{BN}}}_{\circledThree} + \underbrace{\frac{\sqrt{d} L \tau R}{\sqrt{BN}}}_{\circledFour} + \underbrace{\frac{d \Delta R}{\tau \sqrt{BN}}}_{\circledFive}  
        \nonumber \\
        & \quad \quad +  \underbrace{L \tau R}_{\circledSix} + \underbrace{\frac{d \Delta R}{\tau}}_{\circledSeven} + \underbrace{L \tau^2 N}_{\circledEight} + \underbrace{\frac{d^2 \Delta^2 N}{\tau^2 L}}_{\circledNine}.
        \nonumber 
   \end{align}
    \begin{flushleft}
        \textit{ Proof of the Theorem \ref{th:AZO_SGD}}. 
    \end{flushleft}
    \textbf{From term $\circledOne$}, we find the number of iterations $N$ required for Algorithm \ref{alg:AZO_SGD} to achieve $\varepsilon$-accuracy:
    \begin{align}
        \circledOne: \quad \frac{L R^2}{ N^2} \leq \varepsilon \quad & \Rightarrow \quad N \geq \sqrt{\frac{L R^2}{\varepsilon}};
        \nonumber \\
         N &= \mathcal{O}\left( \sqrt{\frac{L R^2}{\varepsilon}} \right). \label{eq:proof_iterations}
    \end{align}
    \textbf{From terms $\circledTwo$ and $\circledThree$}, we find the batch size $B$ required to achieve optimality in iteration complexity $N$: 
    \begin{align}
        &\circledTwo: \quad \frac{L R^2 }{BN} \leq \varepsilon \quad \Rightarrow \quad B \geq \frac{L R^2}{\varepsilon N} 
        \overset{\eqref{eq:proof_iterations}}{=} \mathcal{O}\left( \sqrt{\frac{L R^2}{\varepsilon}}\right);
        \nonumber \\
        &\circledThree: \quad \frac{\sqrt{d} \sigma_{*}R}{\sqrt{BN}} \quad  \Rightarrow \quad B \geq \frac{d \sigma^2_* R^2}{\varepsilon^2 N} 
        \overset{\eqref{eq:proof_iterations}}{=}  \mathcal{O}\left( \frac{d \sigma^2_* R}{\varepsilon^{3/2} L^{1/2}} \right);
        \nonumber \\
        & \quad \quad \quad B = \max \left\{ \mathcal{O}\left( \sqrt{\frac{L R^2}{\varepsilon}}\right),  \mathcal{O}\left( \frac{d \sigma^2_* R}{\varepsilon^{3/2} L^{1/2}} \right)\right\}. \label{eq:proof_batch_size}
    \end{align}
    \textbf{From terms $\circledFour$, $\circledSix$ and $\circledEight$} we find the smoothing parameter $\tau$:
    \begin{align}
        &\circledFour: \quad \frac{\sqrt{d} L \tau R}{\sqrt{BN}} \leq 
        \varepsilon \quad \Rightarrow \quad \tau \leq \frac{\varepsilon \sqrt{BN}}{\sqrt{d} LR} \overset{\eqref{eq:proof_iterations}, \eqref{eq:proof_batch_size}}{=} \max \left\{ \sqrt{\frac{\varepsilon}{d L}}, \frac{\sigma_*}{L} \right\};
        \nonumber \\
        & \circledSix: \quad L \tau R \leq 
        \varepsilon \quad \Rightarrow \quad \tau \leq \frac{\varepsilon}{L R};
        \nonumber \\
        & \circledEight: \quad L \tau^2 N \leq \varepsilon \quad \Rightarrow \quad \tau \leq \sqrt{\frac{\varepsilon}{L N}} \overset{\eqref{eq:proof_iterations}}{=} \frac{\varepsilon^{3/4}}{L^{3/4} R^{1/2}};
        \nonumber \\
        & \quad \quad \quad \tau \leq \min \left\{ \max \left\{ \sqrt{\frac{\varepsilon}{d L}}, \frac{\sigma_*}{L} \right\}, \frac{\varepsilon}{L R}, \frac{\varepsilon^{3/4}}{L^{3/4} R^{1/2}} \right\} = \frac{\varepsilon}{L R}. \label{eq:proof_smoothing_parameter}
    \end{align}
    \textbf{From the remaining terms $\circledFive$, $\circledSeven$, and $\circledNine$}, we find the maximum allowable level of adversarial noise $\Delta$ that still guarantees the convergence of the Accelareted Zero-Order Stochastic Gradient Descent Method to desired accuracy~$\varepsilon$:
    \begin{align}
        &\circledFive: \quad \frac{d \Delta R}{\tau \sqrt{BN}} \leq \varepsilon \quad \Rightarrow \quad \Delta \leq \frac{\varepsilon \tau \sqrt{B N}}{d R} \overset{\eqref{eq:proof_iterations}, \eqref{eq:proof_batch_size}, \eqref{eq:proof_smoothing_parameter}}{=} \max \left\{ \frac{\varepsilon^{3/2}}{d \sqrt{L} R}, \frac{\varepsilon \sigma_*}{d L R} \right\};
        \nonumber \\
        & \circledSeven: \quad \frac{d \Delta R}{\tau} \leq \varepsilon \quad \Rightarrow \quad \Delta \leq \frac{\varepsilon \tau}{d R} \overset{\eqref{eq:proof_smoothing_parameter}}{=} \frac{\varepsilon^2}{d L R};
        \nonumber \\
        &\circledNine: \quad \frac{d^2 \Delta^2 N}{\tau^2 L} \leq \varepsilon \quad \Rightarrow \quad \Delta \leq \sqrt{\frac{\varepsilon \tau^2 L}{d^2 N}} \overset{\eqref{eq:proof_iterations}, \eqref{eq:proof_smoothing_parameter}}{=} \frac{\varepsilon^{7/4}}{d L^{3/4} R^{3/2}};
        \nonumber \\
        & \quad \quad \quad \Delta \leq \min \left\{ \max \left\{ \frac{\varepsilon^{3/2}}{d \sqrt{L} R}, \frac{\varepsilon \sigma_*}{d L R} \right\}, \frac{\varepsilon^2}{d L R}, \frac{\varepsilon^{7/4}}{d L^{3/4} R^{3/2}} \right\} = \frac{\varepsilon^2}{d L R}. 
        \label{eq:proof_noise_level}
    \end{align}
    In this way, the Accelareted Zero-Order Stochastic Gradient Descent (AZO-SGD) Method (see Algorithm \ref{alg:AZO_SGD}) achieves $\varepsilon$-accuracy: $\expect{f(x_{N}^{ag}) - f^*} \leq \varepsilon$ after 
        \begin{equation*}
            N = \mathcal{O}\left( \sqrt{\frac{L R^2}{\varepsilon}} \right), \quad T = N \cdot B = \max \left\{ \mathcal{O}\left( \frac{LR^2}{\varepsilon} \right), \mathcal{O}\left( \frac{d \sigma_*^2 R^2}{\varepsilon^{2}} \right) \right\}
        \end{equation*}
        number of iterations \eqref{eq:proof_iterations}, total number of gradient-free oracle calls \eqref{eq:proof_batch_size} and at
        \begin{equation*}
            \Delta \leq \frac{\varepsilon^2}{d L R^2}
        \end{equation*}
        the maximum level of noise \eqref{eq:proof_noise_level} with smoothing parameter $\tau = \frac{\varepsilon}{L R}$ \eqref{eq:proof_smoothing_parameter}. \\
    \qed

\end{document}